\documentclass{article}

\usepackage{amsmath}
\usepackage{amsthm}
\usepackage{amsfonts}
\usepackage{hyperref}
\usepackage{amssymb}
\newtheorem{theorem}{Theorem}
\newtheorem{lemma}{Lemma}
\newtheorem{corollary}{Corollary}
\newtheorem*{acknowledgment}{Acknowledgment}
\theoremstyle{definition}

\usepackage{authblk}

\begin{document}

\title{Polynomials Consisting of Quadratic Factors with Roots Modulo Any Positive Integer}
\author{Bhawesh Mishra}
\affil[]{Department of Mathematics\\ The Ohio State University, Columbus}

\maketitle

\begin{abstract}
We give an infinite family of polynomials that have roots modulo every positive integer but fail to have rational roots. Each polynomial in this family is made up of monic quadratic factors that do not have linear term. The proofs in this note are accessible to anyone with basic knowledge of undergraduate elementary number theory. 
\end{abstract}

\section{Introduction.} 

We concern ourselves with polynomials $f$ with integer coefficients such that $f(x)\equiv 0 \hspace{1mm}(\text{mod}\hspace{1mm} m)$ is solvable for every positive integer $m$. If a polynomial $f$ has an integer root then $f$ clearly has roots modulo every positive integer. However, there exist polynomials that have roots modulo every positive integer but do not have any rational root. Such polynomials provide counterexamples to the local-global principle in number theory (see \cite[pp. 99 -- 108]{Gou} for more details on the local-global principle).  

Given a cube-free integer $n \neq 1$, Hyde, Lee, and Spearman proved in \cite{HLS} that 

\begin{equation*}
    g(x) = (x^{3} - n) (x^{2} + 3)
\end{equation*}
has no rational roots but has roots modulo every positive integer if and only if $n \equiv 1 \hspace{1mm}(\text{mod}\hspace{1mm} 9)$ and all prime factors of $n$ are equivalent to $1$ modulo $3$.

It is well known that if $p, q$ are distinct odd primes such that $p \equiv q \equiv 1($mod $4)$ and $p$ is a square modulo $q$, then the polynomial 
\begin{equation*}
    h(x) = (x^{2} - p) (x^{2} - q) (x^{2} - pq)
\end{equation*}
has no rational roots but has roots modulo every positive integer (see \cite[pp 139--140]{Jones}). Let $c$ and $d$ be square-free integers not equal to $1$, let $c_{1} = \frac{c}{\text{gcd}(c,d)}$ and $d_{1} = \frac{d}{\text{gcd}(c,d)}$. Hyde and Spearman obtained necessary and sufficient conditions for the polynomial \begin{equation*}
    p(x) = (x^{2} - c) (x^{2} - d) (x^{2} - c_{1}d_{1})
\end{equation*}
to have a root modulo every integer but have no rational root \cite{HS}. We obtain necessary and sufficient conditions for polynomials, made up of any number of similar quadratic factors, to have roots modulo every integer but have no rational root.

Note that the equation $f(x) \equiv 0 \hspace{1mm}(\text{mod}\hspace{1mm} m)$
is solvable for every positive integer $m$ if and only if the equation $f(x) \equiv 0 \hspace{1mm}(\text{mod}\hspace{1mm} p^{b})$ is solvable for each prime number $p$ and each positive integer $b$. This is a consequence of the Chinese remainder theorem. 

Given a prime $p$ and an integer $n$, we denote the Legendre symbol of $n$ with respect to $p$ by $\big(\frac{n}{p}\big)$. When $p \nmid n$, $\big(\frac{n}{p}\big) = +1 $ if $n$ is a square modulo $p$ and $\big(\frac{n}{p}\big) = -1$ otherwise. When $p \mid n$, $\big(\frac{n}{p}\big) = 0$.  Given a prime $p$ and a non-zero integer $l$,  $p^{a} \mid\mid l$ will denote that $p^{a} (a \geq 1)$ is the highest power of $p$ dividing $l$. Our main result is the following theorem.

\begin{theorem}
Let $n \geq 3$, $a_{1}, a_{2}, \ldots , a_{n}$ be distinct nonzero square-free integers, none of which is $1$. Then the polynomial $f(x) = (x^{2} - a_{1}) (x^{2} - a_{2}) \cdots (x^{2} - a_{n})$ has roots modulo every positive integer if and only if the following conditions are satisfied:

\begin{enumerate}
    
    \item There exists $T \subseteq \{1, 2, \ldots , n\}$ of odd cardinality such that:
    
    \begin{enumerate}
    \item the product $\prod_{j \in T} a_{j}$ is a perfect square;

    \item for every $j \in T$ and for every odd prime $p$ dividing $a_{j}$, there exists $i \in \{1, \ldots , n\}$, $i \neq j$, such that $\big(\frac{a_{i}}{p}\big) = +1 $ . 
    
    \end{enumerate}
    
    \item One of the $a_{i}$ is of the form $8m + 1$ for some $m \in\mathbb{Z}$ and $m \neq 0$.
    
\end{enumerate}
\end{theorem}

We note that the polynomial $f(x) = \prod_{i=1}^{n} (x^{2} - a_{i})$ in Theorem $1$ cannot have a rational root, since none of the $a_{i}$ is a perfect square.

\section{Results on Quadratic Residues.}
In this section, we collect some results on quadratic reciprocity that will be used to prove Theorem $1$. 
\begin{enumerate}

    \item Let $p$ be an odd prime, $e \geq 1$, and $a \in\mathbb{Z}$ such that $p \nmid a$. Then $a$ is a square modulo $p^{e}$ if and only if $\big(\frac{a}{p}\big) = +1.$\\
    Since $p \nmid a$, the proof of this statement is an immediate consequence of Hensel's Lemma applied to the polynomial $q(x) = (x^{2} - a)$ (See \cite[p. 135]{Jones} for a proof).  
    
    \item \begin{enumerate}
    
        \item Let $a \in \mathbb{Z}$ be an odd number. Then $a$ is a square modulo $2^{i}$ for every $i \geq 1$ if and only if $a = 8m + 1$ for some $m \in\mathbb{Z}$(See \cite[p. 136]{Jones}) .

       \item Let $a$ be a square-free integer not equal to $1$. Then $a$ is a square modulo $2^{i}$ for every $i \geq 1$ if and only  $a = 8m + 1$ for some $m \in\mathbb{Z}$ and $m \neq 0$ (See \cite[Lemma 2.5]{HS}).
       
       \begin{proof}
       If $a$ is odd then the result follows from 2(a). On the other hand, if $a$ is even, then $a \equiv 2\hspace{1mm} (\text{mod } 4)$ because it is square-free. Since $2$ is not a square modulo $2^{2}$, $a$ cannot be a square modulo $2^{2}$. 
       \end{proof}
    
    \end{enumerate}

\end{enumerate}

\section{Some Useful Lemmas.}

We will also use the following lemma, a proof of which appears in \cite{FR}. 

\begin{lemma}
Let $a_{1}, a_{2}, ... , a_{n}$ be finitely many nonzero integers. Then the following conditions are equivalent:

\begin{enumerate}
    \item For each prime $p$ that does not divide $\prod_{i=1}^{n} a_{i}$ , at least one of the $a_{i}$ is a square modulo $p$. 
    
    \item There exists $T \subset \{1, 2, \cdots , n\}$ of odd cardinality such that $\prod_{j\in T} a_{j} $ is a perfect square. 
\end{enumerate}
\end{lemma}

\begin{lemma}
Let $p$ be an odd prime and $a \in\mathbb{Z}$ be a square-free integer. If $\big(\frac{a}{p}\big) \neq +1$, then $a$ cannot be a square modulo $p^{2}$. 
\end{lemma}

\begin{proof}
If $p \nmid a$ then $\big(\frac{a}{p}\big) \neq +1$, along with result $1$ on quadratic reciprocity, gives that $a$ is not a square modulo $p^{2}$. 

If $p \mid a$ then $p \mid\mid a$ because $a$ is square-free. Assume that $a$ is a square modulo $p^{2}$, i.e., $x^{2} \equiv a \hspace{1mm} (\text{mod}\hspace{1mm} p^{2})$ for some $x \in\mathbb{Z}$. Then we have $p^{2} \mid (x^{2} - a)$, i.e., $p \mid (x^{2} - a)$. 

Since $p \mid a$ and $p \mid (x^{2} - a)$, we have that $p \mid x^{2}$ implying $p^{2} \mid x^{2}$. However,\\ $p^{2} \mid (x^{2} - a)$ and $p^{2} \mid x^{2}$ gives that $p^{2} \mid a$, contradicting that $a$ is square-free. Therefore, $a$ cannot be a square modulo $p^{2}$.  
\end{proof}

\begin{lemma}
Let $p$ be a prime number, $k,n \in\mathbb{N}$, and $f(x) =  \prod_{i=1}^{n} (x^{2} - a_{i}) \in\mathbb{Z}[x]$. If $(x^{2} - a_{i}) \equiv 0 \hspace{1mm} (\text{mod}\hspace{1mm} p^{k})$ is not solvable for any $1 \leq i \leq n$, then $f(x) \equiv 0 \hspace{1mm} (\text{mod}\hspace{1mm} p^{kn})$ is not solvable.  
\end{lemma}

\begin{proof}
For the sake of contradiction, assume that $f(x) \equiv 0 \hspace{1mm} (\text{mod}\hspace{1mm} p^{kn})$ is solvable for some $x \in\mathbb{Z}$, i.e., $p^{kn} \mid (x^{2} - a_{1}) \cdots (x^{2} - a_{n})$. Then we must have $p^{k} \mid (x^{2} - a_{j})$, for some $j \in \{1, 2, \ldots , n\}$. 

However $p^{k} \mid (x^{2} - a_{j})$ implies that $(x^{2} - a_{j}) \equiv 0 \hspace{1mm} (\text{mod}\hspace{1mm} p^{k})$ is solvable, a contradiction to the fact that $(x^{2} - a_{i}) \equiv 0 \hspace{1mm} (\text{mod}\hspace{1mm} p^{k})$ is not solvable for any $i$. Therefore, we have the result. 
\end{proof}

\section{Proof of Theorem 1.}

\subsection{Proof of Necessity:} The necessity of the condition $1(a)$ immediately follows from Lemma $1$. Now we will show that if condition $1(b)$ or condition $2$ fails then $f(x)$ fails to have root modulo some integer.\vspace{1mm}
   
Assume that the condition $1(b)$ in Theorem $1$ fails: if $T$ is a subset of $\{1, 2, \ldots , n\}$ of odd cardinality for which $\prod_{j \in T} a_{j}$ is a perfect square, then there exists $j \in T$ and an odd prime $p \mid a_{j}$ such that $\big(\frac{a_{i}}{p}\big) \neq +1 $ for any $i \neq j$. Since each $a_{i}$ is square-free, Lemma $2$ implies that for any $i \neq j$, $(x^{2} - a_{i}) \equiv 0 \hspace{1mm} (\text{mod}\hspace{1mm} p^{2})$ is not solvable. Similarly, $(x^{2} - a_{j}) \equiv 0 \hspace{1mm} (\text{mod}\hspace{1mm} p^{2})$ is not solvable because $\big(\frac{a_{j}}{p}\big) = 0 \neq +1 $. \vspace{1mm}

Since $(x^{2} - a_{i}) \equiv 0 \hspace{1mm} (\text{mod}\hspace{1mm} p^{2})$ is not solvable for any $1 \leq i \leq n$, Lemma $3$ for $k = 2$ implies that the equation $f(x) \equiv 0 \hspace{1mm} (\text{mod}\hspace{1mm} p^{2n})$ is not solvable.\vspace{1mm}
   
If the condition $2$ in the statement of Theorem 1 fails, then none of the $a_{i}$ is equal to $(8m + 1)$ for any nonzero $m \in\mathbb{Z}$. Since each of the $a_{i}$ is square-free, result $2(b)$ on quadratic reciprocity implies for any $i \in \{1, 2, \ldots , n\}$, $(x^{2} - a_{i}) \equiv 0 \hspace{1mm}(\text{mod}\hspace{1mm}8)$ is not solvable. Then applying Lemma $3$ for $p = 2$ and $k = 3$ gives that the congruence $f(x) \equiv 0 \hspace{1mm} (\text{mod}\hspace{1mm} 8^{n})$ is not solvable.

\subsection{Proof of Sufficiency:}As a consequence of the Chinese remainder theorem, it suffices to show that $f(x) \equiv 0 \hspace{1mm}(\text{mod}\hspace{1mm} p^{b})$ is solvable for every prime $p$ and every integer $b \geq 1$.\vspace{1mm}
   
If $p$ is an odd prime that does not divide $\prod_{j \in T} a_{j}$, an application of Lemma $1$ for $\{a_{j}\}_{j \in T}$ implies that for some $j_{0} \in T$, $a_{_{j_{0}}}$ is a square modulo $p$. Since $p \nmid \prod_{j \in T} a_{j}$, result $1$ on quadratic reciprocity implies that  $a_{_{j_{0}}}$ is a square modulo every positive power of $p$. Therefore, $f(x) \equiv 0 \hspace{1mm}(\text{mod}\hspace{1mm} p^{b})$ is solvable for every integer $b \geq 1$.\vspace{1mm}

If $p$ is an odd prime dividing $\prod_{j \in T} a_{j}$, then $p \mid a_{j}$ for some $j \in T$. Condition $1(b)$ of Theorem $1$ ensures that there exists $i \neq j$ such that $\big(\frac{a_{i}}{p}\big) = +1$. Since $\big(\frac{a_{i}}{p}\big) = +1$, we have that $p \nmid a_{i}$. Therefore, $a_{i}$ is a square modulo all powers of $p$, using result $1$ on quadratic reciprocity. Hence, $f(x) \equiv 0 \hspace{1mm}(\text{mod}\hspace{1mm} p^{b})$ is solvable for every integer $b \geq 1$.\vspace{1mm}
   
If $p = 2$ then condition $2$ of the theorem ensures that for some $k \in \{1, 2, \ldots , n\}$ and for some nonzero integer $m$, $a_{k} = 8m + 1$. Then $(x^{2} - a_{k}) \equiv 0 \hspace{1mm} (\text{mod}\hspace{1mm} 2^{b})$ is solvable for every $b \geq 1$, by result $2(b)$ on quadratic reciprocity. Therefore, $f(x) \equiv 0 \hspace{1mm}(\text{mod}\hspace{1mm} 2^{b})$ is solvable for every integer $b \geq 1$. 

\section{Corollaries and Examples.}
The first corollary gives a necessary and sufficient condition for the polynomial $(x^{2} - p) (x^{2} - q) (x^{2} - pq)$ to have roots modulo every positive integer. Here $p$ and $q$ are distinct odd primes.

\begin{corollary}
Let $p$ and $q$ be distinct odd primes. The polynomial\begin{equation*} f(x) = (x^{2} - p) (x^{2} - q) (x^{2} - pq)\end{equation*} has roots modulo every positive integer if and only if $\big( \frac{p}{q} \big) = \big( \frac{q}{p} \big) = +1$.
\end{corollary}

\begin{proof}
Note that $p \times q \times pq = (pq)^{2}$ and that one of $p$ or $q$ or $pq$ has to be of the form $(8m + 1)$ for $0 \neq m \in \mathbb{Z}$. So the result follows by applying Theorem $1$ to $a_{1} = p$, $a_{2} = q$, and $a_{3} = a_{1}a_{2} = pq$. 
\end{proof}

When $p$ and $q$ are distinct primes such that $p \equiv q \equiv 1 \hspace{1mm}(\text{mod}\hspace{1mm} 4)$, $\big(\frac{p}{q}\big) = +1 $ if and only if $\big(\frac{q}{p}\big) = +1$. Hence the polynomial $h(x) = (x^{2} - p) (x^{2} - q) (x^{2} - pq)$ has roots modulo every integer when $p \equiv q \equiv 1 \hspace{1mm}(\text{mod}\hspace{1mm} 4)$ and $\big(\frac{p}{q}\big) = +1$, as mentioned in the Introduction.

\begin{corollary}
Let $c$ and $d$ be two square-free integers not equal to $1$, $c_{1} = \frac{c}{\text{gcd}(c,d)}$, and $d_{1} = \frac{d}{\text{gcd}(c,d)}$. Then $f(x) = (x^{2} - c) (x^{2} - d) (x^{2} - c_{1}d_{1})$ has roots modulo every integer if and only if the following conditions are satisfied:
\begin{enumerate}
    \item For each odd prime $p$ dividing $c$, $\big(\frac{d}{p}\big) = +1  $ or $\big(\frac{c_{1}d_{1}}{p}\big) = +1$ and for each odd prime $p$ dividing $d$, $\big(\frac{c}{p}\big) = +1$ or $\big(\frac{c_{1}d_{1}}{p}\big) = +1$.
    
    \item At least one of $c$, $d$ and $c_{1}d_{1}$ is of the form $(8m + 1)$ for some nonzero integer $m$. 
\end{enumerate}
\end{corollary}

Necessary and sufficient conditions for the polynomial in above corollary to have a root modulo every positive integer were obtained by Hyde and Spearman in \cite{HS}. In our Corollary $2$, there are fewer conditions to check compared to the result obtained in \cite{HS}.

\begin{proof}
Apply Theorem $1$ for three square-free integers $a_{1} = c$, $a_{2} = d$, and $a_{3} = c_{1} d_{1}$ and note that $a_{1} \times a_{2} \times a_{3} = c \times d \times c_{1}d_{1} = (rc_{1}d_{1})^{2}$, where $r =$ gcd $(c, d)$.
\end{proof}

It is instructive to see some examples of application of Theorem $1$. The polynomials listed below are obtained from different values of $n$ and square-free integers $a_{1}, a_{2}, \ldots , a_{n}$ in Theorem $1$. 

\begin{enumerate}
\item $(x^{2} - 13) (x^{2} - 17) (x^{2} - 221)$ has roots modulo every integer because

$13 \times 17 \times 221 = (221)^{2}$, $\big(\frac{13}{17}\big) = +1$, $\big(\frac{17}{13}\big) = +1$, and $17 = 8 (2) + 1$.\vspace{2mm}

\item $(x^{2} - 7) (x^{2} - 11) (x^{2} - 19) (x^{2} - 31) (x^{2} - 209)$ has roots modulo every integer because

$ 11 \times 19 \times 209 = (209)^{2}$,
$\big(\frac{31}{11}\big) = +1$, $\big(\frac{7}{19}\big) = +1$, and $209 = 8 (26) + 1$.\vspace{2mm}

\item $(x^{2} - 7) (x^{2} - 11) (x^{2} - 19) (x^{2} - 31) (x^{2} - 45353)$ has roots modulo every integer because

$7 \times 11 \times 19 \times 31 \times 45353 = (45353)^{2}$,
$\big(\frac{11}{7}\big) = +1$, $\big(\frac{31}{11}\big) = +1$, $\big(\frac{7}{19}\big) = +1$, $\big(\frac{19}{31}\big) = +1$, and $45353 = 8 (5669) + 1$.\vspace{2mm}

\end{enumerate}

\begin{acknowledgment}
This research was partly supported by the NSF, under grant DMS-1812028.
\end{acknowledgment}

\bibliographystyle{amsalpha}
\bibliography{Intersective_2}

\end{document}